\documentclass[12pt]{amsart}

\usepackage{tikz}

\usepackage{mathrsfs}
\usepackage[colorlinks=true,urlcolor=blue, citecolor=red,linkcolor=blue,linktocpage,pdfpagelabels, bookmarksnumbered,bookmarksopen]{hyperref}
\usepackage[hyperpageref]{backref}
\usepackage{cleveref}
\usepackage{xcolor}
\usepackage{amsthm} 
\usepackage{latexsym,amsmath,amssymb}
\usepackage{accents}
\usepackage[colorinlistoftodos,prependcaption,textsize=tiny]{todonotes}
\usepackage{a4wide}
\usepackage{soul}
\usepackage{mathtools} 
\usepackage{xparse} 
\usepackage{enumitem}

\definecolor{indigo}{rgb}{0.29, 0.0, 0.51}
\definecolor{p1}{gray}{0.4}
\definecolor{p2}{gray}{0.6}
\definecolor{p3}{gray}{0.98}
\definecolor{p4}{gray}{0.8}
\definecolor{p5}{gray}{0.9}

plus 6pt minus 12pt
\newcommand{\HL}{\mathrm{HL}}

\def\eps{\varepsilon}

\def\vp{\varphi}

\newcommand{\cM}{\mathcal{M}}
\newcommand{\cN}{\mathcal{N}}
\newcommand{\cH}{\mathcal{H}}

\def\B{{B}}

\def\n{{\mathcal M}}

\def\N{{\mathbb N}}
\def\n{{\mathcal{M}}}
\def\n{{\mathcal N}}

\def\S{{\mathbb S}}

\DeclarePairedDelimiter{\floor}{\lfloor}{\rfloor}

\renewcommand{\div}{{\rm div}}
\newcommand{\dv}{\operatorname{div}}

\newtheorem{theorem}{Theorem}

\newtheorem{lemma}[theorem]{Lemma}
\newtheorem{corollary}[theorem]{Corollary}
\newtheorem{proposition}[theorem]{Proposition}

\newtheorem{remark}[theorem]{Remark}
\newtheorem{definition}[theorem]{Definition}
\newtheorem{example}[theorem]{Example}

\newtheorem{openproblem}[theorem]{Open Problem}

\newcommand{\dif}{\,\mathrm{d}}
\newcommand{\dx}{\dif x}

\newcommand{\dd}{\dif}

\newcommand{\R}{\mathbb{R}}

\newcommand{\pl}{\partial}
\newcommand{\ale}{\lesssim}

\newcommand{\brac}[1]{\left (#1 \right )}
\newcommand{\abs}[1]{\left |#1 \right |}

\newcommand{\ov}[1]{\overline{#1}}
\newcommand{\la}{\left\langle}
\newcommand{\ra}{\right\rangle}

\newcommand{\barint}{
\rule[.036in]{.12in}{.009in}\kern-.16in \displaystyle\int }

\newcommand{\barcal}{\mbox{$ \rule[.036in]{.11in}{.007in}\kern-.128in\int $}}

\def\mvint_#1{\mathchoice
          {\mathop{\vrule width 6pt height 3 pt depth -2.5pt
                  \kern -8pt \intop}\nolimits_{\kern -3pt #1}}%
          {\mathop{\vrule width 5pt height 3 pt depth -2.6pt
                  \kern -6pt \intop}\nolimits_{#1}}%
          {\mathop{\vrule width 5pt height 3 pt depth -2.6pt
                  \kern -6pt \intop}\nolimits_{#1}}%
          {\mathop{\vrule width 5pt height 3 pt depth -2.6pt
                  \kern -6pt \intop}\nolimits_{#1}}}

\numberwithin{theorem}{section} \numberwithin{equation}{section}

\def\XXint#1#2#3{{\setbox0=\hbox{$#1{#2#3}{\int}$}
     \vcenter{\hbox{$#2#3$}}\kern-.5\wd0}}

\let\latexchi\chi
\makeatletter
\renewcommand\chi{\@ifnextchar_\sub@chi\latexchi}
\newcommand{\sub@chi}[2]{
  \@ifnextchar^{\subsup@chi{#2}}{\latexchi^{}_{#2}}%
}
\newcommand{\subsup@chi}[3]{
  \latexchi_{#1}^{#3}%
}
\makeatother

\title[Regularity of $p$-harmonic maps]{Regularity of minimizing $p$-harmonic maps into spheres and sharp Kato inequality}

\author{Katarzyna Mazowiecka}
\address[Katarzyna Mazowiecka]{
Institute of Mathematics,%
University of Warsaw,
Banacha 2,
02-097 Warszawa, Poland}
\email{k.mazowiecka@mimuw.edu.pl}

\author{Micha\l{} Mi\'{s}kiewicz}
\address[Micha\l{} Mi\'{s}kiewicz]{
Institute of Mathematics,%
University of Warsaw,
Banacha 2,
02-097 Warszawa, Poland
\newline
\&
Institute of Mathematics, Polish Academy of Sciences, \'{S}niadeckich 8, 00-656 Warszawa, Poland}
\email{m.miskiewicz@mimuw.edu.pl}

\begin{document}

\begin{abstract}
We study regularity of minimizing $p$-harmonic maps $u \colon B^3 \to \S^3$ for $p$ in the interval $[2,3]$. For a long time, regularity was known only for $p = 3$ (essentially due to Morrey \cite{Morrey}) and $p = 2$ (Schoen--Uhlenbeck \cite{SU3}), but recently Gastel \cite{Gastel19} extended the latter result to $p \in [2,2+\frac{2}{15}]$ using a version of Kato inequality. Here, we establish regularity for a small interval $p\in [2.961,3]$ by combining Morrey's methods with Hardt and Lin's Extension Theorem \cite{HLp}. We also improve on the other result by obtaining regularity for $p \in [2,p_0]$ with $p_0 = \frac{3+\sqrt{3}}{2} \approx 2.366$. In relation to this, we address a~question posed by Gastel and prove a sharp Kato inequality for $p$-harmonic maps in two-dimensional domains, which is of independent interest. 
\end{abstract}
\sloppy

\maketitle

\tableofcontents
\sloppy

\section{Introduction}

The object of our study are minimizing $p$-harmonic maps, which are defined as minimizers of the Dirichlet $p$-energy
\begin{equation}\label{def:p-energy}
 E_p(u) \coloneqq \int_{\cM} |\nabla u|^p \dif V \qquad \text{for } u \in W^{1,p}(\cM,\cN)
\end{equation}
with fixed boundary data $u|_{\pl\cM}$, where $\cM$, $\cN$ are smooth Riemannian manifolds, $\cM$ has a non-empty boundary and $\cN$ is closed. The regularity theory of these maps started with studying the more classical case $p = 2$, in which they are simply called \emph{minimizing harmonic maps}. The classical result of Morrey \cite{Morrey} states that all minimizing harmonic maps are regular if $n = 2$, $n$ being the domain dimension. In the seminal paper \cite{SU1} Schoen and Uhlenbeck showed that in higher dimensions, these maps are regular in the interior outside of a singular set of Hausdorff dimension at most $n-3$. Moreover, the singular set is discrete if $n=3$. 

The result of Schoen and Uhlenbeck \cite{SU1} is sharp: for $3$-dimensional domains minimizing harmonic maps may have point singularities, as shown by the ``hedgehog'' map
\[
 u \colon B^3 \to \S^2, \quad u(x) = \frac{x}{|x|};
\]
see, e.g., \cite{BCL} for the proof of minimality. Similar examples are available for general $n \ge 3$, showing that the singular set may indeed be $(n-3)$-dimensional. 

Since singularities may appear, one has to be careful with the statement of the problem. Postponing a detailed discussion until Section \ref{s:preliminaries}, let us mention that we assume $\cN \subset \R^d$ to be isometrically embedded, and $W^{1,p}(\cM,\cN)$ is defined as the set of all $u \in W^{1,p}(\cM,\R^d)$ satisfying $u(x) \in \cN$ a.e. Moreover, since the regularity considerations are local in nature, we may assume for simplicity --- as is customary in the literature --- that $\cM = \Omega \subset \R^n$ is a~Euclidean domain. 

Although the cited result is sharp in general, regularity properties of harmonic maps depend crucially on the geometry of the target manifold $\cN$. For example, solutions are regular if $\cN$ has non-positive sectional curvature \cite{SU1}, see also related earlier works \cite{Eells-Sampson,HKW}. In a~subsequent paper Schoen and Uhlenbeck \cite{SU3} studied the case when the target manifold is a sphere $\cN =\S^{k}$. Their result, combined with a further improvement of Okayasu \cite{Okayasu94}, gives:
\begin{theorem}[\cite{SU3}, \cite{Okayasu94}]
 Any minimizing harmonic map $u \colon B^n \to \S^k$ is regular whenever $n \le d(k)$, where 
 \[
  d(k) = 
  	\begin{cases}
                  k &\text{ for } 2 \le k\le 5\\
                  5 &\text{ for } 6 \le k \le 9\\
                  6 &\text{ for } k \ge 10.
  	\end{cases}
 \]
\end{theorem}
If we specialize this result to the case $n = k$, we deduce that all minimizing harmonic maps 
\[
 u \colon B^n \to \S^n
\]
must be regular for $n = 2,3,4,5$. It is an open problem whether minimizing harmonic maps $u \colon B^6 \to \S^6$ are continuous. In higher dimensions, J\"{a}ger and Kaul \cite{JK} proved that the (singular) equatorial map defined by
\[
v \colon B^n \to \S^n, \qquad v(x)\coloneqq \brac{\frac{x}{|x|},0}
\]
is a minimizing harmonic map if and only if $n\ge 7$ \cite{JK}. In particular, the equatorial map is \emph{not} a minimizing harmonic map for $n=6$. 

\medskip

Let us return to $p$-harmonic maps for general exponent $p$. Here, regularity theory for minimizing $p$-harmonic maps was developed in analogy to the case $p = 2$, at least to some extent. Hardt and Lin \cite{HLp} showed that minimizing $p$-harmonic maps are regular outside of a singular set, which in general has Hausdorff dimension at most $n - \floor{p} - 1$, and is discrete in the case $n = \floor{p}+1$. In particular, these maps are regular when $p \ge n$; note that this special case can also be obtained by an adaption of Morrey's classical argument. 

This result may be improved for special target manifolds, as shown by Nakauchi \cite{Nakauchi01} and Xin--Yang \cite{XY}. In the case of the target manifold being a sphere we have
\begin{theorem}[\cite{Nakauchi01,XY}]\label{th:nakauxy}
 Any minimizing $p$-harmonic map $u \colon B^n \to \S^k$ is regular whenever $n \le d(k,p)$, where 
 \[
  d(k,p) = 
	\begin{cases}
                  \lfloor p \rfloor + 1 &\text{ for } p+2\le k <p+4\\
                  \lfloor p \rfloor + 2 &\text{ for } p+4\le k <p+6\\
                  \min\{\lfloor p \rfloor + 3,\lambda(p)\} &\text{ for } p+6\le k <p+8\\
                  \min\{\lfloor p \rfloor + 4,\lambda(p)\} &\text{ for } p+8\le k,	
	\end{cases}
 \]
 and $\lambda(p) = p + 2\sqrt{1-\frac{p-2}{(p-1)^2}} + \frac{2}{p-1}$.
\end{theorem}
Unfortunately, this estimate is not optimal and in particular it leaves an open question whether minimizing $p$-harmonic maps are regular for $u \colon B^n \to \S^n$ and $p \in [n-1,n)$.

In a recent paper Gastel \cite{Gastel19} was exploiting the connection between the theory of Cosserat micropolar elasticity and $p$-harmonic maps. He showed that the only obstacle for the regularity of minimizers in a geometrically nonlinear Cosserat model for micropolar elasticity of continua is the possible existence of non-constant $p$-minimizing tangent maps from $B^3$ into $\S^3$. 

In this setting Gastel improved \Cref{th:nakauxy} by showing that all minimizing $p$-harmonic maps from $B^3$ to $\S^3$ must be regular when $2 \le p \le \frac{32}{15} \approx 2.133$. Combined with the regularity theory of Hardt and Lin \cite{HLp}, it implies that minimizing $p$-harmonic maps from $B^3$ to $\S^3$ must be regular for all $p$ in $[2,\frac{32}{15}] \cup [3,\infty)$. Let us mention that Gastel's results were extended to the case of stable-stationary solutions by Li and Wang \cite{Li-Wang22}.

Our main goal is to study the regularity of minimizing $p$-harmonic maps from $B^3$ to $\S^2$. We obtain: 
\begin{theorem}\label{th:main}
 Let $p_0 = \frac{3+\sqrt{3}}{2}\approx 2.366$ and $p_1 = 3-\frac{17-24\ln(2)}{3^{3/2}2^{5/6}} \approx 2.961$. If $p\in[2,p_0] \cup [p_1,3]$, then any minimizing $p$-harmonic map $u \colon B^3 \to \S^3$ is regular. 
\end{theorem}
This result improves on the previous work of Gastel and Hardt--Lin. Interestingly, it leaves a gap $(p_0,p_1)$ of length approximately $0.6$, in which regularity is still an open problem. It is our belief that that regularity holds for these values of $p$ as well, and we provide certain motivation by showing a certain ``counterexample candidate'' fails (see \Cref{le:notmini}). However, the problem of regularity in this interval may be as difficult as determining whether minimizing harmonic mappings $u\colon B^6\to\S^6$ must be regular.

In the following, we decided to break \Cref{th:main} into two separate statements: \Cref{th:regularitycloseto2} for $p \in [2,p_0]$ and \Cref{th:regularity-near-n-1} for $p \in [p_1,3]$. The first reason for this is that the two regimes require very different techniques. And the second, that in both cases the method of proof actually leads to a more general result. 

{\bf Proof outline for \Cref{th:main}.}

Regularity for $p = 3$ follows along the lines of Morrey's proof \cite{Morrey}. Given a minimizing $u \colon B^3 \to \S^3$ and considering its restriction $v \coloneqq u\big\rvert_{\pl B_r}$ to a sphere ($r \in [\frac 12,1]$), we may extend $v$ to a map $\ov{v} \colon B_{r} \to \S^3$ with $p$-energy controlled by that of $v$. For a generic slice, the latter is controlled by the $p$-energy of $u$ on $B_{1} \setminus B_{1/2}$ and hence minimality gives 
\[
\int_{B_{1/2}} |\nabla u|^p \dx
\le \int_{B_{r}} |\nabla u|^p \dx
\le \int_{B_{r}} |\nabla \ov{v}|^p \dx
\ale \int_{B_{1} \setminus B_{1/2}} |\nabla u|^p \dx.
\]
The hole-filling trick now leads to polynomial decay of $E_p(u)$ on smaller and smaller balls. In other words, $\nabla u$ lies in a Morrey space $L^{p,\lambda}$ with some small $\lambda > 0$, which implies continuity of $u$. 

The crucial part of this argument is the ability to extend $v$ with energy control, and Morrey's construction of this extension does not apply for $p < 3$. However, there is another construction available in the literature: Hardt-Lin's extension theorem \cite[Theorem 6.2]{HLp}. This more refined topology-sensitive tool applies with a uniform constant for all $p \in [2,3]$, thus enabling us to proceed with the argument. In result, we obtain $\nabla u \in L^{p,\lambda}$ for each such $p$ \emph{with uniform $\lambda > 0$}. If only $p \in (3-\lambda,3]$, such regularity is enough to conclude the continuity of $u$. The argument we give in \Cref{s:regularitynear3} is slightly different: we decided to exploit the notion of tangent maps, which simplifies the exposition and lets us keep track of constants more easily. 

The method we employ in the regime $p \in [2,\frac{3+\sqrt{3}}{2}]$ dates back to Schoen and Uhlenbeck \cite{SU3}. For simplicity, we will only discuss the classical example $p = 2$ here and mention other authors' contributions when necessary. The proof follows by excluding the existence of non-constant tangent maps, i.e., minimizers with symmetry $u(\lambda x) = u(x)$ for $\lambda > 0$. Such maps are determined by their restriction to $\S^2$, and Schoen and Uhlenbeck derived the following estimates: 
\[
\frac 13 \int_{\S^2} |\nabla u|^4 - \frac 14 \int_{\S^2} |\nabla u|^2  
\le \int_{\S^2} |\nabla |\nabla u||^2, \quad 
\int_{\S^2} |\nabla^2 u|^2
\le \frac 12 \int_{\S^2} |\nabla u|^4  - \int_{\S^2} |\nabla u|^2. 
\]
One can combine these two by the simple pointwise inequality $|\nabla |\nabla u||^2 \le |\nabla^2 u|^2$, leading to an inequality of the form 
\[
A \int_{\S^2} |\nabla u|^4 + B \int_{\S^2} |\nabla u|^2 \le 0.
\]
If only $A,B > 0$, this would prove that $u$ is necessarily constant, and indeed an analogous argument works for maps into $\S^4,\S^5,\S^6,\ldots$. However, here we have $A = - \frac 16 < 0$ and additional insight is needed: Okayasu \cite{Okayasu94} noted that for harmonic maps, the pointwise inequality is improved to $|\nabla |\nabla u||^2 \le \frac 12 |\nabla^2 u|^2$ --- what we call a Kato inequality. That way, one obtains $A = \frac{1}{12}$ and $B = \frac{1}{4}$, which completes the proof (note that Schoen and Uhlenbeck covered the case of maps $B^3 \to \S^3$ too, but with entirely different methods). 

Gastel \cite{Gastel19} followed the same line of reasoning for $p > 2$, and to this end he proved a non-optimal Kato inequality $|\nabla |\nabla u||^2 \le \kappa(p) |\nabla^2 u|^2$, with $\kappa(p) \xrightarrow{p \to 2} \frac 12$. That way, regularity is still obtained for $p$ close to $2$ (one can check numerically that the argument covers $p \le \frac{32}{15} \approx 2.133$). Our proof differs from Gastel's in two ways. First, in the context of maps $\S^2 \to \S^3$ we obtain a Kato inequality with the \emph{optimal} constant, which by itself would enlarge the regularity interval to $p \le \frac{5+\sqrt{17}}{4} \approx 2.281$. Second, we note a possible --- quite technical in nature --- improvement in one of the other inequalities, which enables us to cover all $p \le \frac{3+\sqrt{3}}{2} \approx 2.366$. 

{\bf The paper is organized as follows.} In \Cref{s:preliminaries} we review the notions of $p$-harmonic maps that we will be using. As a motivation to study the regularity problem, in \Cref{s:motivation} we prove that the natural candidate for a singular $p$-harmonic map --- the equatorial map --- is not minimizing $p$-harmonic when $p \in [2,3)$. In \Cref{s:regularitynear3} we show how the extension theorem of Hardt and Lin \cite{HLp} can be used to obtain regularity of maps into $\cN = \S^3$ for exponents $p$ close to $3$. Next, in \Cref{s:regularitynear2} we modify Gastel's argument \cite{Gastel19}, which is itself based on the original idea of Schoen and Uhlenbeck \cite{SU3}, and establish regularity for $p \in [2,\frac{3+\sqrt{3}}{2}]$. Finally, in \Cref{s:Kato} we discuss sharp Kato inequality for $p$-harmonic maps in $2$-dimensional domains, as well as its relation to the regularity problem, and comment on the inequality in higher dimensions. We close the paper with \Cref{s:constants}, in which explicit constant estimates are given for the inequalities used in \Cref{s:regularitynear3}.

{\bf Acknowledgments.}
\begin{itemize}
 \item The project is co-financed by the Polish National Agency for Academic Exchange within Polish Returns Programme - BPN/PPO/2021/1/00019/U/00001 (KM).
 \item The project is co-financed by National Science Centre grant 2022/01/1/ST1/00021 (KM).
 \item MM was supported by the NCN Sonatina grant no. 2020/36/C/ST1/00050.
\end{itemize}

\section{Preliminaries}\label{s:preliminaries}
\subsection{Weakly, stationary, and stable $p$-harmonic maps}
In the world of non-smooth solutions, it is important to distinguish between different classes of $p$-harmonic maps. We will only consider minimizing and stable stationary $p$-harmonic maps in the sequel, but we also define weakly $p$-harmonic maps for a complete picture. To keep it simple, we restrict here to maps from an open domain $\Omega \subset \R^n$ into $\S^k \subset \R^{k+1}$.

\begin{definition}
\label{def:weaklypharmonic}
\label{def:stationary}
\label{def:stable}
\label{def:minimizing}
A (weakly) $p$-harmonic map $u \in W^{1,p}(\Omega,\S^k)$ is defined as a critical point of the $p$-energy with respect to variations in the range, i.e.,
\begin{equation}\label{eq:p-harmonic}
 \frac{d}{dt}\bigg\rvert_{t=0} E_p(u_t) = 0 
 \quad \text{where} \quad u_t = \frac{u+t\vp}{|u+t\vp|}
\end{equation}
for all $\vp \in C_c^\infty(\Omega,\R^{k+1})$. Equivalently, as a solution the Euler--Lagrance equation 
\[
-\dv(|\nabla u|^{p-2} \nabla u) = |\nabla u|^p u \quad \text{weakly}.
\]
Moreover, it is called: 
\begin{itemize}

\item \emph{stationary} if it is also a critical point with respect to variations of the domain, i.e.,
\begin{equation}
 \frac{d}{dt}\bigg\rvert_{t=0} E_p(u_t) = 0 
 \quad \text{where} \quad u_t(x) = u(x+t\xi(x))
\end{equation}
for all $\xi \in C^\infty_c(\Omega,\R^n)$.

\item \emph{stable} if it the energy has a non-negative second variation at $u$: 
\begin{equation}\label{eq:stability-def}
 \frac{d^2}{dt^2}\bigg\rvert_{t=0} E_p(u_t) \ge 0 
 \quad \text{where} \quad u_t = \frac{u+t\vp}{|u+t\vp|},
\end{equation}
or equivalently
\begin{equation}\label{eq:p-stability}
 \int_{\Omega}|\nabla u|^{p-2}\brac{|\nabla \varphi|^2-|\nabla u|^2|\varphi|^2} \dx + (p-2)\int_{\Omega} |\nabla u|^{p-4}\brac{\nabla u \cdot \nabla \varphi}^2 \dx \ge 0
\end{equation}
for all $\varphi \in C_c^\infty(\Omega, \R^{k+1})$ satisfying $\vp(x) \cdot u(x)=0$ a.e.

\item \emph{minimizing} if $E_p(u) \le E_p(v)$ for all $v \in W^{1,p}(\Omega,\S^k)$ with the same trace on $\pl \Omega$. 

\end{itemize}
\end{definition}

We have the following inclusions between the different classes of $p$-harmonic maps:
\[
 \text{minimizing}
 \subset \text{stable stationary} 
 \subset \text{stationary} 
 \subset \text{weak}.
\]
We note here that in the case $k > 2p-1$ stable $p$-harmonic maps are quasiminimizers \cite[Proposition 2]{Hong}.

\subsection{Tangent maps}
\begin{definition}
A $p$-harmonic map $u \colon \R^n \to \cN$ is called \emph{a tangent map} if it is $0$-homogeneous, meaning that $u(\lambda x) = u(x)$ for all $\lambda > 0$ and $x \in \R^n$. Moreover, we say it is a $p$-minimizing tangent map if it minimizes the $E_p$-energy \eqref{def:p-energy} on compact subsets of $\R^n$. 
\end{definition}

Most of the regularity theory of minimizing harmonic maps relies on the study of minimizing tangent maps. The following theorem of Hardt and Lin (which is a generalization of \cite[Theorem IV]{SU1}) allows us to deduce regularity by ruling out the existence of non-constant tangent maps.
\begin{theorem}[{\cite[Theorem 4.5]{HLp}}]\label{th:HL-tangentmaps}
Assume that the only $p$-minimizing tangent maps from $\R^n$ into $\cN$ are constant maps. Then all minimizing $p$-harmonic maps $u \colon \Omega \to \mathcal N$ on $n$-dimensional domains are regular. 
\end{theorem}

\begin{remark}\label{rem:HL-tangentmapsforstablestationary}
Using the compactness lemma for stable-stationary $p$-harmonic maps \cite[Lemma 4.3]{Hong-Wang99} 
the latter theorem may be improved to the following:

Assume $p<3$. If every stable-stationary $p$-harmonic tangent map from $\R^3$ to $\S^3$ is constant then every stable-stationary $p$-harmonic map $u \colon B^3 \to \S^3$ is regular. 
\end{remark}

\section{Motivation: the equatorial map}\label{s:motivation}
In this section we will study the equatorial map $v\colon B^n \to \S^n$ defined by
\[
 v(x) \coloneqq \brac{\frac{x}{|x|},0}.
\]
It is easy to verify that $v$ is weakly $p$-harmonic for any $p < n$. For $p \in (n-2\sqrt{n-1},n)$, we will show that it is not a minimizing $p$-harmonic map because it is not stable. As mentioned in the introduction, the map $v$ above is an important counterexample to regularity in many cases. Hence, our result motivates an attempt to prove that minimizing $p$-harmonic maps $B^n \to \S^n$ in general are regular for $p$ close to $n$. 

\begin{lemma}\label{le:notmini}
 The equatorial map $v$ is unstable for $p \in (n-2\sqrt{n-1},n)$.
\end{lemma}
In particular, $v$ is unstable for every $p \in [2,3]$ when $n \le 6$. Let us compare this result with minimality of the ``hedgehog'' map $\frac{x}{|x|}\colon B^n \to \S^{n-1}$, for which we have: 
\begin{theorem}\label{th:hedgehogmin}
 The map $\frac{x}{|x|}\colon B^n \to \S^{n-1}$ is a minimizing $p$-harmonic map for every $1 \le p < n$.
\end{theorem}
\Cref{th:hedgehogmin} is due to J\"{a}ger--Kaul \cite{JK} in the case $p=2$, $n\ge 7$; Brezis--Coron--Lieb \cite{BCL} for $n=3$, $p=2$; Coron--Gulliver \cite{CG} when $p=1,2,\ldots,n-1$; Lin \cite{Lin-remark} for $n\ge 3$, $p=2$; Hardt--Lin--Wang \cite{HLW} in the case $p\in[n-1,n)$; Hong \cite{Hong-minimality} when $1<p\le n-1$.

Let us also recall that for $p=2$ we have
\begin{theorem}[{\cite[Theorem 2, \textsection 3]{JK}}]\label{thm:JK-stability}
The equatorial map $\brac{\frac{x}{|x|},0}\colon B^n \to \S^n$ is minimizing harmonic if and only if $n\ge 7$. 
\end{theorem}

\begin{proof}[Proof of \Cref{le:notmini}]
We follow the proof of \cite[Theorem 2]{JK}. 

Assume on the contrary that $v$ satisfies the stability inequality \eqref{eq:p-stability} for all $\vp \in C_c^\infty(\B^n,\R^{n+1})$ satisfying $v(x) \cdot \vp(x) = 0$ a.e. Choosing 
\[
 \vp(x) \coloneqq \eta(|x|)e_{n+1},
\]
where $\eta \colon[0,1]\to \R$, we easily have $v(x) \cdot \vp(x) = 0$. Additionally we have $\nabla v \cdot \nabla \varphi =0$ as $\varphi$ is radial and $v$ is constant on the radii. Since $|\nabla v(x)|^2 = \frac{n-1}{|x|^2}$, the stability inequality for our choice of $\varphi$ would give 
\begin{equation}
 \int_{B^n} |x|^{2-p}\brac{ (\eta'(|x|))^2 - \frac{n-1}{|x|^2}\eta^2(|x|)} \dif x \ge 0,
\end{equation}
which after integrating over spheres $S_r$ is equivalent to
\begin{equation}\label{eq:stabilityforueta}
 \int_0^1 r^{n-p+1}\brac{(\eta'(r))^2 - \frac{n-1}{r^2}\eta^2(r)} \dif r\ge 0.
\end{equation}
In order to obtain a contradiction, we will construct a non-zero function $\eta \colon [0,1] \to \R$ satisfying the ODE 
\begin{equation}\label{eq:etachoice}
\begin{cases}
  \eta''(r) + \frac{n-p+1}{r}\eta'(r) + \frac{n-1-\varepsilon}{r^2}\eta(r) = 0 
  & \text{ for } r \in (r_0,1), \\
  \eta(r_0) = \eta(1) = 0.
\end{cases}
\end{equation}
for some small parameters $\eps(p,n), r_0(p,n) > 0$. Assuming that we have a solution $\eta$ of \eqref{eq:etachoice} extended by zero to $[0,r_0]$, integration by parts leads to 
\begin{equation}
\begin{split}
  \int_0^1 & r^{n-p+1}\brac{(\eta'(r))^2 - \frac{n-1}{r^2}\eta^2(r)} \dif r \\
  &= - \int_{r_0}^1 r^{n-p+1}\eta(r)\brac{\eta''(r)+\frac{n-p+1}{r}\eta'(r) + \frac{n-1}{r^2}\eta(r)} \dif r\\
  &= - \eps \int_{r_0}^1 r^{n-p-1}{\varepsilon}\eta^2(r) \dif r < 0, 
\end{split}
\end{equation}
which contradicts \eqref{eq:stabilityforueta}.

It remains to prove that there is a non-zero solution to \eqref{eq:etachoice}.  Using the substitution $\zeta(t) = \eta(e^t)$ we arrive at the ODE 
\begin{equation}\label{eq:zetaequation}
 \zeta''(t) +(n-p) \zeta'(t) + (n-1-\eps) \zeta(t) = 0,
\end{equation}
whose solutions are explicit: 
\[
\zeta(t) = 
\begin{cases}
	c_1 e^{\lambda_1 t} + c_2 e^{\lambda_2 t} & \text{ if } \Delta > 0, \\
	e^{\lambda t} (c_1  + c_2 t ) & \text{ if } \Delta = 0, \\
	e^{\lambda t} (c_1 \cos(\mu t) + c_2 \sin(\mu t)) & \text{ if } \Delta < 0,
\end{cases}
\]
depending on the sign of the determinant $\Delta_\eps = (n-p)^2-4(n-1-\eps)$. Note however that the boundary conditions $\zeta(\ln r_0) = \zeta(0) = 0$ force $\zeta$ to be identically zero when $\Delta_\eps \ge 0$. 

Thus, we restrict to the range $p \in (n-2\sqrt{n-1},n+2\sqrt{n-1})$, in which we have $\Delta_0 < 0$. Fixing small $\eps(n,p) > 0$, we still have $\Delta_\eps < 0$, and thus $\zeta(t) = e^{\lambda t} \sin(\mu t)$ is a solution of \eqref{eq:zetaequation} for some suitable $\lambda, \mu$ ($\lambda$ may be small but positive). We now choose $r_0 \in (0,1)$ so that $\ln r_0$ is a negative multiple of $\frac{2\pi}{\lambda}$, in which case $\zeta$ satisfies the boundary conditions and produces a solution of \eqref{eq:etachoice}: 
\[
 \eta(r) = r^{\frac{p-n}{2}} \sin(\sqrt{-\Delta_\eps/4} \cdot \ln r).
\]
This finishes the proof.
\end{proof}

\section{Regularity for \texorpdfstring{$p$}{p} close to 3}\label{s:regularitynear3}
As a special case, the regularity theory of Hardt and Lin \cite[Corollary 2.6]{HLp} for minimizing $p$-harmonic maps shows that any minimizing $n$-harmonic map $u\in W^{1,n}(B^n,\mathcal N)$ is regular. However, this particular regularity theorem can also be proved in a way similar to how Morrey \cite{Morrey} established regularity of minimizing harmonic maps in dimension $2$, see also \cite[Chapter 2.1]{Lin-Wang-book}. Exploiting the methods of both these proofs, we obtain the main result of this section.

\begin{theorem}
\label{th:regularity-near-n-1}
 There is an $\eps(n) > 0$ such that every minimizing $p$-harmonic map $u\in W^{1,p}(B^n, \S^n)$ with $n-\eps < p \le n$ is regular. 
 
 In the case $n=3$ we have 
$\eps = \frac{17-24\ln 2}{3^{3/2}2^{5/6}}\approx 0.0394$, in particular if $p \in [2.961,3]$ then every minimizing $p$-harmonic map $u \colon B^3\to\S^3$ is regular.
\end{theorem}

\begin{remark}
As can be seen by inspecting the proof, the only ingredient that depends on the target manifold is Hardt and Lin's extension theorem \cite[Theorem 6.2]{HLp}, which holds in greater generality. Hence, the same regularity theorem is true for any $n$-connected target manifold $\cN$ in place of $\S^n$, with $\eps(n,\cN) > 0$ now depending also on $\cN$. More generally:  if $n\ge4$, for any manifold $\n$ such that $\pi_1(\n)$ is finite and $\pi_2(\n)\simeq \ldots \simeq \pi_n(\n)\simeq \{0\}$, see \cite{MVS-tr}.
\end{remark}

\begin{proof}
\textsc{Step 1}: We begin with the following Claim:

For any $p \in [n-1,n]$ and $w \in W^{1,p}(\S^{n-1},\S^n)$, there is an extension $v\in W^{1,p}(B^n,\S^n)$ such that $v \big\rvert_{\pl B^n} = w \big\rvert_{\pl B^n}$ in the trace sense and 
\begin{equation}\label{eq:goodextensionintospheres}
\int_{\B^n} |\nabla v|^p \dx \le C_{\text{ext}}(p,n) \int_{\S^n} |\nabla w|^p \dif \cH^{n-1},
\end{equation}
with constant $C(p,n)$ depending on both $p$ and $n$, but bounded uniformly for $p \in [n-1,n]$. 

Indeed, we may use a variant of the trace theorem to extend $w \colon \S^{n-1}\to \S^n$ to $\bar v \in W^{1,p}(B^n,\R^{n+1})$ satisfying
\begin{equation}\label{eq:traceinequality1}
 \int_{B^n} |\nabla \bar v|^p \dx \le C_T(p,n) \int_{\S^{n-1}} |\nabla w|^p \dif \cH^{n-1}.
\end{equation}
Next, since $p \le n$ we have $\pi_{j}(\S^{n})=\{0\}$ for all $j \le \floor{p-1}$, and we may use Hardt and Lin's extension theorem \cite[Theorem 6.2]{HLp}\footnote{This Theorem has been reproved several times, see, e.g.,   \cite[Proof of 2.3. Lemma]{HardtKinderlehrerLin1986}, \cite[Lemma A.1]{HardtKinderlehrerLinstable}, \cite[Lemma 2.2]{BPVS14}, \cite[Lemma 4.7]{Hopper} , \cite[Section 5.1]{MMS}.} to construct $v\in W^{1,p}(B^n,\S^{n})$ with $v\big\rvert_{\partial B^n} = \bar v\big\rvert_{\partial B^n} = u$ in the trace sense and such that

\begin{equation}\label{eq:HLextension}
 \int_{B^n} |\nabla v|^p \dx \le C_{\text{HL}}(p,n) \int_{B^n} |\nabla \bar v|^p \dx.
\end{equation}
Combining \eqref{eq:traceinequality1} and \eqref{eq:HLextension} we obtain \eqref{eq:goodextensionintospheres} with $C_{\text{ext}}(p,n) = C_T(p,n)C_{\HL}(p,n)$.

\textsc{Step 2}: By Theorem \ref{th:HL-tangentmaps}, it is enough to restrict our attention to $0$-homogeneous maps. In other words, we will show that any $p$-minimizing tangent map $u \colon \R^n \to \S^n$ has to be constant.
\[
\int_{\B^n} |\nabla u|^p \dx 
= \int_0^1 r^{n-1-p} \dd r \int_{\S^{n-1}} |\nabla_T u|^p \dif\mathcal H^{n-1} 
= \frac{1}{n-p} \int_{\S^{n-1}} |\nabla_T u|^p \dif\mathcal H^{n-1}. 
\]
We use now \textsc{Step 1} to construct a map $v\in W^{1,p}(B^n,\S^n)$ with $v\big\rvert_{\partial B^n} = u\big\rvert_{\partial B^n}$. Combining the minimality of $u$ with \eqref{eq:goodextensionintospheres} we obtain 
\[
\frac{1}{n-p} \int_{\S^{n-1}} |\nabla_T u|^p \dif\mathcal H^{n-1}
=\int_{\B^n} |\nabla u_0|^p \dx
\le \int_{\B^n} |\nabla v|^p \dx
\le C_{\text{ext}}(p,n) \int_{\S^{n-1}} |\nabla_T u|^p \dif \mathcal H^{n-1}.
\]
If $u$ is non-constant, we necessarily have $\frac{1}{n-p} \le C_{\text{ext}}(p,n)$. But if it so happens that $C_{\text{ext}}(p,n) < \frac{1}{n-p}$, the resulting contradiction shows that minimizing $p$-harmonic maps cannot have singularities. The theorem then follows for $\eps(n) = \left( \sup_{p\in[n-1,n]} C_{\text{ext}}(p,n) \right)^{-1}$. 

\textsc{Step 3}. Explicit estimates in the case $n = 3$.

By \Cref{th:trace} in dimension 3 we have explicit estimates for the constant $C_T(p,3) = C_T(p)$, which gives $\sup_{p \in [2,3]} C_T(p) = C_T(3) = \frac{3^{3/2}}{2^{13/6}}$. By \Cref{th:HL-constantestimate} we obtain estimates for $C_{\HL}(p,3) = C_{\HL}(p)$: we have $\sup_{p \in [2,3]} C_{\HL}(p) = C_{\HL}(3) = \frac{8}{17-24\ln 2}$. Thus,
\[
 3-\frac{1}{C_{\text{ext}}(p,3)} = 3-\frac{1}{C_T(3)C_{\HL}(3)}\le 2.961.
\]
This comparison implies that whenever $p \ge 2.961$, regularity holds for all minimizing $p$-harmonic maps $u \in W^{1,p}(B^3,\S^3)$.
\end{proof}

\section{Regularity for \texorpdfstring{$p$}{p} close to 2}\label{s:regularitynear2}
In this section we show how to improve the regularity result of Gastel \cite[Section 6]{Gastel19}. As noticed by Gastel with improved Kato type inequality one could improve the results of Nakauchi \cite{Nakauchi01} and Xin--Yang \cite{XY} in the same way that Okayasu \cite{Okayasu94} improved the result of Schoen--Uhlenbeck \cite{SU3}. As explained earlier the only obstacle to regularity is the existence of non-constant tangent maps. Combining a stability inequality, a corollary of Bochner formula, and a mixed Cauchy--Schwarz--Kato\footnote{In \Cref{rem:mixedinsteadofkato} we point out that using this inequality in place of the optimal Kato inequality (which we prove later in \Cref{lem:Kato-optimal}) gives a better constant.} inequality we prove that any stable stationary tangent $p$-harmonic map must be constant.  

\begin{theorem}\label{th:regularitycloseto2}
If $p \in [2,\frac{3+\sqrt{3}}{2}]$, then any stable stationary $p$-harmonic map $u \colon B^3 \to \S^3$ is regular. 
\end{theorem}
Our proof will show that any tangent map $u \colon \R^3 \to \S^3$ must be constant. As outlined in the introduction, there are three ingredients: 

\begin{lemma}\label{lem:ingredients}
For a stable stationary $p$-harmonic tangent map $u \colon \R^3 \to \S^3$, the following estimates hold: 
\begin{enumerate}[label=(\alph*)]
\item\label{it:stability} Stability inequality: \\
\begin{gather*}
\int_{\S^2} |\nabla u|^{p-4} \left( 3 |\nabla u|^2 |\nabla |\nabla u||^2 + (p-2) |\la \nabla u, \nabla |\nabla u| \ra|^2 \right) \\
\ge (3-p) \int_{\S^2} |\nabla u|^{p+2} - \frac 34 (3-p)^2 \int_{\S^2} |\nabla u|^p; 
\nonumber
\end{gather*}
\item\label{it:bochner} Bochner formula corollary: \\
\begin{equation*}
\int_{\S^2} |\nabla u|^{p-2} (|\nabla^2 u|^2 + (p-2) |\nabla |\nabla u||^2) \le \frac{1}{2} \int_{\S^2} |\nabla u|^{p+2} - \int_{\S^2} |\nabla u|^p;
\end{equation*}
\item\label{it:mixedC-S-K} Mixed Cauchy--Schwarz--Kato inequality: 
\begin{equation*}
3 |\nabla |\nabla u||^2 + (p-2) \left |\la \tfrac{\nabla u}{|\nabla u|}, \nabla |\nabla u| \ra \right |^2 
\le \frac{3}{p} \left( |\nabla^2 u|^2 + (p-2) |\nabla |\nabla u||^2 \right)
\end{equation*}
at all points where $\nabla u \neq 0$.
\end{enumerate}
\end{lemma}

\begin{remark}
The constant $\frac 3p$ in \Cref{lem:ingredients}\ref{it:mixedC-S-K} is sharp in a weak, pointwise sense: the function $u \colon \R^2 \to \R$ given by $u(x,y) = x+xy$ solves the $p$-harmonic equation at $(0,0)$ and yields equality. One can also modify it to a map $u \colon \S^2 \to \S^3$ with similar properties, just by using exponential coordinates.
\end{remark}

Proofs of \Cref{lem:ingredients} \ref{it:stability} and \ref{it:bochner} were essentially given by Nakauchi. For \ref{it:stability}, one has to follow the proof in \cite[Lem.~3]{Nakauchi96} but omitting the step where $|\la \nabla u, \nabla |\nabla u| \ra|$ is bounded by $|\nabla u| |\nabla |\nabla u||$. One side effect of this alteration is an improved constant in front of the $p$-energy: $\frac{3}{4}(3-p)^2$ instead of $\frac{1}{4}(p+1)(3-p)^2$. Note however that this particular improvement is irrelevant in our applications. For \ref{it:bochner}, the derivation in \cite[Lem.~1]{Nakauchi96} applies (see also \cite[Lem.~2]{Nakauchi01}), but this time one omits the step $|\nabla|\nabla u|| \le |\nabla^2 u|$.

\begin{remark}\label{rem:mixedinsteadofkato}
Part \Cref{lem:ingredients}\ref{it:mixedC-S-K} compares the left-hand sides of \ref{it:stability} and \ref{it:bochner} with the best possible constant $\frac{3}{p}$. If one applied separately Cauchy--Schwarz $|\la \nabla u, \nabla |\nabla u| \ra| \le |\nabla u| |\nabla |\nabla u||$ and optimal Kato inequality $2 |\nabla|\nabla u||^2 \le |\nabla^2 u|^2$, the constant would be $\frac{p+1}{p}$. However, one can inspect the proofs of these two inequalities to see that equality cannot hold in both at the same time. Because of this, there is an advantage in considering an ad hoc mixed Cauchy--Schwarz--Kato inequality. 
\end{remark}

Before proving \Cref{lem:ingredients}\ref{it:mixedC-S-K}, let us put all the pieces together and see how regularity of $p$-harmonic follows from Lemma \ref{lem:ingredients}

\begin{proof}[Proof of Theorem \ref{th:regularitycloseto2}]
\leavevmode

Assume that $u\colon \R^3\to\S^3$ is a stable stationary $p$-harmonic tangent map. By \Cref{rem:HL-tangentmapsforstablestationary} it suffices to prove that $u$ must be a constant map.

Applying Lemma \ref{lem:ingredients} to $u$ --- comparing the left-hand sides of \ref{it:stability} and \ref{it:bochner} via \ref{it:mixedC-S-K} --- produces an inequality involving just the integrals of $|\nabla u|^p$ and $|\nabla u|^{p+2}$. After rearranging, it reads 
\[
\tfrac{3}{2}(4-p(3-p)^2) \int_{\S^2} |\nabla u|^p 
+ (2p(3-p)-3) \int_{\S^2} |\nabla u|^{p+2} \le 0.
\]
For $p \in [2,3]$, it is easily seen that $p(3-p)^2 \le 3 \cdot 1^2$ and hence the coefficient $\tfrac{3}{2}(4-p(3-p)^2)$ is always positive. However, the coefficient $2p(3-p)-3$ is nonnegative only for $p$ between $2$ and $\frac{3+\sqrt{3}}{2}$. For $p$ in this interval, we can conclude that $\int_{\S^2} |\nabla u|^p \le 0$ and hence $u$ is constant. 
\end{proof}

\begin{proof}[Proof of \Cref{lem:ingredients}\ref{it:mixedC-S-K}]\leavevmode

Let us consider a point $q \in \S^2$ for which $\nabla u(q) \neq 0$. We will use exponential coordinates around $q \in \S^2$ and $u(q) \in \S^3$, thus reducing the problem to the flat case. Indeed, the $p$-harmonic system in coordinates takes the form \cite[p.~iii]{Hun94}
\[
- \frac{1}{\sqrt{\det \gamma}} \pl_\beta \left( \sqrt{\det \gamma} \left( \gamma^{ij} g_{ab} u_i^a u_j^b \right)^{\frac{p-2}{2}} \gamma^{ij} u_i^c \right) 
= \left( \gamma^{ij} g_{ab} u_i^a u_j^b \right)^{\frac{p-2}{2}} \gamma^{ij} \Gamma_{ab}^{c} u_i^a u_j^b
\quad \text{for } c=1,2,3,
\]
where $u^a$ are components of $u$, $u_i^a$ their partial derivatives $\partial_i u^a$, $\gamma_{ij}$ and $g_{ab}$ are coefficients of metric tensors $\gamma$ on $\S^2$ and $g$ on $\S^3$ ($\gamma^{ij}$ denoting the inverse matrix), $\Gamma_{ab}^c$ are Christoffel symbols on $\S^3$, and the summation over $i,j,a,b$ is implicitly understood by Einstein convention. However, in exponential coordinates we have $\gamma_{ij} = \delta_{ij}$, $g_{ab} = \delta_{ab}$ up to first order at $q$, hence also all $\Gamma_{ab}^c$ vanish at that point. This simplifies the equation at \emph{one particular point} $q \in \S^2$ to 
\begin{equation}\label{eq:p-harmonicflat}
0 = \sum_{i=1}^2 u_{ii}^a + (p-2) \sum_{i,j=1}^2\sum_{b=1}^3 u_i^a u_j^b u_{ij}^b,
\quad \text{for all } a=1,2,3,
\end{equation}
which is the same as the equation for $p$-harmonic maps $u \colon \R^2 \to \R^3$. Since we are interested in a pointwise inequality, we shall only employ \eqref{eq:p-harmonicflat} at $q$. 

When using exponential coordinates, we may still choose coordinates in the tangent spaces $T_q \S^2$ and $T_{u(q)} \S^3$ to our advantage, which we will now do. First, we may assume that $|\nabla u(q)| = 1$, as in general we could just work with $\frac{\nabla u(q)}{|\nabla u(q)|}$. In order to simplify the coordinate calculations, we shall choose a basis of $T_q \S^2$ for which the partial derivatives $u_1,u_2 \in T_{u(q)} \S^3$ are orthogonal. To see why it is possible, start with any basis and note that $\la u_1,u_2 \ra$ changes sign as one rotates it by $\pi/2$. Then, we may choose an orthonormal basis $e_1,e_2,e_3 \in T_{u(q)} \S^3$ for which $u_1 = \alpha_1 e_1$ and $u_2 = \alpha_2 e_2$. With our choice of bases and notation, \eqref{eq:p-harmonicflat} takes the convenient form 
\begin{equation}\label{eq:p-harmonicflat-coordinates}
\begin{cases}
0 = u_{11}^1+u_{22}^1 + (p-2)\alpha_1(\alpha_1 u_{11}^1 + \alpha_2 u_{12}^2) \\
0 = u_{22}^2+u_{11}^2 + (p-2)\alpha_2(\alpha_2 u_{22}^2 + \alpha_1 u_{12}^1) \\
0 = u_{11}^3+u_{22}^3.
\end{cases}
\end{equation}
We may use the first two equations to express $u_{22}^1$, $u_{11}^2$ in terms of other second order derivatives; the third equation will not be exploited. The expressions appearing in \Cref{lem:ingredients}\ref{it:mixedC-S-K} can be written as
\begin{equation}\label{eq:mixedCSK-computations}
\begin{split}
1 = |\nabla u(q)|^2 &= \alpha_1^2+\alpha_2^2,\\
|\nabla|\nabla u|(q)|^2 &= (\alpha_1 u_{11}^1 + \alpha_2 u_{12}^2)^2 + (\alpha_2 u_{22}^2 + \alpha_1 u_{12}^1)^2, \\
\left |\la \tfrac{\nabla u(q)}{|\nabla u(q)|}, \nabla |\nabla u|(q) \ra \right |^2 
&= \alpha_1^2 (\alpha_1 u_{11}^1 + \alpha_2 u_{12}^2)^2 + \alpha_2^2 (\alpha_2 u_{22}^2 + \alpha_1 u_{12}^1)^2
\end{split}
\end{equation}
and
\begin{equation}\label{eq:mixedCSK-computations2}
\begin{split}
|\nabla^2 u(q)|^2 &= \brac{(u_{11}^1)^2 + 2 (u_{12}^2)^2 + (u_{22}^1)^2} + \brac{(u_{22}^2)^2 + 2 (u_{12}^1)^2 + (u_{11}^2)^2)} + \sum_{i,j=1}^2 (u_{ij}^3)^2\\
&= I_1 + I_2 + I_3.
\end{split}
\end{equation}
In order to prove \Cref{lem:ingredients}\ref{it:mixedC-S-K} it suffices to show that 
\begin{equation}
\label{eq:mixed-CS-Kato-raw}
\begin{split}
3&(\theta_1 x + \theta_2 y)^2 + (p-2) \theta_1^2 (\theta_1 x + \theta_2 y)^2 \\
&\le \frac 3p \Big( x^2 + 2y^2 + (x + (p-2) \theta_1 (\theta_1 x + \theta_2 y))^2 + (p-2) (\theta_1 x + \theta_2 y)^2 \Big), 
\end{split}
\end{equation}
holds for all $x,y \in \R$ and $\theta_1,\, \theta_2\in[0,1]$ such that $\theta_1^2 + \theta_2^2 = 1$.

Indeed, taking $x = u_{11}^1$, $y = u_{12}^2$, $\theta_1 = \alpha_1$, $\theta_2=\alpha_2$, we see from the equation \eqref{eq:p-harmonicflat-coordinates} that $-u_{22}^1 = x + (p-2)\theta_1(\theta_1 x + \theta_2 y)$, and then \eqref{eq:mixed-CS-Kato-raw} produces \emph{part} of \Cref{lem:ingredients}\ref{it:mixedC-S-K}, i.e., a~comparison with the term $I_1$ involving Hessian components $u_{11}^1$, $u_{12}^2$, $u_{22}^1$. To obtain the other part, which involves the term $I_2$ with components $u_{22}^2$, $u_{12}^1$, $u_{11}^2$, one just switches the indices and take $x=u_{22}^2$, $y=u_{12}^1$, $\theta_1 = \alpha_2$, $\theta_2=\alpha_1$. We obtain a slightly stronger version of \Cref{lem:ingredients}\ref{it:mixedC-S-K} --- without the terms $I_3$ of $|\nabla^2 u|^2$ on the right-hand side.

Thus, we are left with proving the elementary inequality \eqref{eq:mixed-CS-Kato-raw}. This becomes easier in rotated variables: $z = \theta_1 x + \theta_2 y$ and $w = \theta_2 x - \theta_1 y$ in place of $x$ and $y$. After some rearranging and multiplying both sides by $\frac{p}{2(p-2)}$ (note that $p>2$ in our case), \eqref{eq:mixed-CS-Kato-raw} becomes 
\[
0 
\le p \theta_1^2 z^2 + 3 \theta_1\theta_2 zw + \frac{3}{p-2} w^2.
\]
Looking at the right-hand side as a quadratic form in $z,w$, it is now straightforward to check that the discriminant is nonpositive for all choices of $\theta_1,\theta_2 \in [0,1]$ and $p \in [2,3]$. This finishes the proof of \eqref{eq:mixed-CS-Kato-raw} and thus also \Cref{lem:ingredients}\ref{it:mixedC-S-K}. 
\end{proof}

\section{Sharp Kato inequality}\label{s:Kato}
As discussed earlier, Kato inequalities are an important tool in studying regularity of $p$-harmonic maps. The progress made in Theorem \ref{th:regularitycloseto2} --- as compared to Gastel's methods and results \cite{Gastel19} --- is due to sharp Kato inequality in two-dimensional domains, proved below in \Cref{lem:Kato-optimal} (however, with the ad hoc estimate of \Cref{lem:ingredients} \ref{it:mixedC-S-K} the final result may be improved further). As observed by Gastel, improvements in Kato inequalities in higher dimensions will lead to new regularity results for minimizing $p$-harmonic maps. 

In this section, we prove sharp Kato inequality for $p$-harmonic maps in two-dimensional domains, for all $p \in (1,\infty)$. We also share some general remarks on the problem and point out to certain open problems. Let us start with defining the optimal constant. 

\begin{definition}
Let $n,d \in \N$, $p \ge 1$, and consider a $C^2$ map $u \colon \R^n \to \R^d$ satisfying the $p$-harmonic map equation $\div(|\nabla u|^{p-2} \nabla u) = 0$ at the point $0 \in \R^n$, with $\nabla u(0) \neq 0$. The optimal Kato constant $\kappa(p,n,d) > 0$ is the largest number such that 
\begin{equation}
\label{eq:df-optimal-Kato}
\kappa(p,n,d) |\nabla |\nabla u|(0)|^2 \le |\nabla^2 u(0)|^2
\end{equation}
for all such maps $u$. 
\end{definition}

\begin{remark}
For $p$-harmonic maps between manifolds $u \colon \cM^n \to \cN^d$ the inequality \eqref{eq:df-optimal-Kato} holds with the same constant $\kappa(p,n,d)$. This is easily seen in exponential coordinates, as in the proof of Theorem \ref{th:regularitycloseto2}. 
\end{remark}

\begin{remark}
Note that the only relevant information about $u$ is encoded in $\nabla u(0)$ and $\nabla^2 u(0)$. Without loss of generality, one can assume these to have length $1$. Thus, the optimal Kato constant can be characterized as 
\begin{equation}
\label{eq:elementary-optimal}
\kappa(p,n,d)^{-1} = \sup_{v,w} \sum_{i} \left( \sum_{j,a} v_j^a w_{ij}^a \right)^2,
\end{equation}
where the supremum is taken over all unit length vectors $v \in \R^n \otimes \R^d$ and unit length $n \times n$ symmetric $\R^d$-valued matrices $w \in \operatorname{Sym}_n(\R^d)$ satisfying the constraint  
\begin{equation}
\label{eq:elementary-p-harmonic}
0 = \sum_{i} w_{ii}^a + (p-2) \sum_{i,j,b} v_i^a v_i^b w_{ij}^b
\quad \text{for all } a=1,\ldots,d.
\end{equation}
This elementary characterization is useful for finding numerical approximations. 
\end{remark}

Let us review what we know about the optimal Kato constant $\kappa(p,n,d)$ for various choices of parameters $p,n,d$.

\begin{proposition}\label{le:kato-constantreview}
Let $\kappa(p,n,d)$ be the constant from \eqref{eq:df-optimal-Kato}. Then the following holds:
\begin{enumerate}[label=(\roman*)]
\item \label{it:p=2} $\kappa(2,n,d) = \frac{n}{n-1}$;
\item \label{it:d=1}$\kappa(p,n,1) = \min(2, 1+\frac{(p-1)^2}{n-1})$;
\item \label{it:n=2}$\kappa(p,2,d) = \kappa(p,2,1) = \min(2, 1+(p-1)^2)$;
\item \label{it:ineqn=3} $\kappa(p,3,2) < \kappa(p,3,1)$ for some choices of $p$, in particular for the critical exponent $p = 1+\sqrt{2}$;
\item \label{it:ge1} $\kappa(p,n,d) \ge 1$ follows trivially, but $\kappa(p,n,d) = 1$ only for $p = 1$; 
\item \label{it:mono} In general, $\kappa(p,n,d)$ is non-increasing in $d$, but $\kappa(p,n,d) = \kappa(p,n,n)$ for all $d \ge n$  (see below).
\end{enumerate}
\end{proposition}

\begin{remark}
Interestingly, $\kappa(p,2,d)$ does not depend on $d$, but $\kappa(p,3,d)$ does! Our numerical simulations suggest that optimal constants for $n \ge 3$ indeed depend on $d$, but differ from $\kappa(p,n,1)$ by only a tiny amount (see \Cref{ex:Kato-example}). Since this difference is miniscule, one could potentially still obtain regularity results for $p$-harmonic maps (in domains of dimensions $\ge 4$) by using these Kato inequalities, even with suboptimal constants. However, the difference makes it hard to imagine a rigorous proof of any close-to-optimal Kato inequality. 
\end{remark}
\begin{openproblem}
 Determine the optimal constant in Kato inequality in other cases than mentioned in \Cref{le:kato-constantreview}, starting with $\kappa(p,3,2)$ and $\kappa(p,3,3)$. 
\end{openproblem}

The rest of the section is dedicated to justifying \Cref{le:kato-constantreview}.

\newcommand{\mmbullet}{$\rightarrow$ \ }
\begin{proof}[Proof of \Cref{le:kato-constantreview}]
\textsc{Ad} \ref{it:p=2}: This is the result of Okayasu, see \cite[Lem.~1.1]{Okayasu94}. Let us remark here that since the $2$-harmonic equation is decoupled, $\kappa(2,n,d)$ is easily seen to be independent of the parameter $d$.

\textsc{Ad} \ref{it:d=1}: This is a result of Chang--Chen--Wei, see \cite[Lem.~5.4]{ChaCheWei16}.

\textsc{Ad} \ref{it:n=2}: This is our main contribution to the problem of the optimal constant in Kato inequality and we prove it separately in \Cref{lem:Kato-optimal}.

\textsc{Ad} \ref{it:ineqn=3}: We show this in \Cref{ex:Kato-example} below.

\textsc{Ad} \ref{it:ge1}: For the trivial inequality $\kappa(p,n,d) \ge 1$, one just uses the Cauchy--Schwarz inequality in \eqref{eq:elementary-optimal}. If for some choice of $v$ and $w$ this happened to be an equality, we would have $w_{ij}^a = \lambda_i v_j^a$ for all $i,j,a$ (and some unit $\lambda \in \R^n$). But this leads to $\lambda_i v_j^a = \lambda_j v_i^a$, which means that all vectors $v_i$ are parallel: $v_i^a = \lambda_i z^a$ for some unit $z \in \R^d$. Then the $p$-harmonic system reduces to $(p-1) z = 0$, which is impossible for $p > 1$. However, it follows from the elementary characterization in \eqref{eq:elementary-optimal}, \eqref{eq:elementary-p-harmonic} that the optimal constant $\kappa(p,n,d)$ is actually achieved for some $v,w$. Hence, $\kappa(p,n,d) > 1$ strictly.

For $p = 1$ we actually have $\kappa(1,n,d) = 1$. This is trivial for $n=d=1$, as in this case any function with non-vanishing derivative satisfies the $1$-harmonic equation; in the general case one just needs to add dummy dimensions in both domain and codomain. 

\textsc{Ad} \ref{it:mono}: Note that any $p$-harmonic map $u \colon \R^n \to \R^d$ can be trivially extended to a $p$-harmonic map $v \colon \R^n \to \R^{d+1}$ by taking $v = (u^1,\ldots,u^d,0)$. Both maps lead to the same constant in \eqref{eq:df-optimal-Kato}. Therefore, the optimal constant $\kappa(p,n,d)$ is non-increasing as a function of $d$. 

On the other hand, given a $p$-harmonic map $u \colon \R^n \to \R^d$ with $d > n$, notice that the image of its differential at $0$ is at most $n$-dimensional. For simplicity, we may assume that it is contained in $\R^n \times \{ 0 \}$; then the map $v = (u^1,\ldots,u^n)$ is easily seen to satisfy the $p$-harmonic map equation at $0$, and 
\[
\kappa(p,n,n) |\nabla |\nabla u|(0)|^2 
= \kappa(p,n,n) |\nabla |\nabla v|(0)|^2 
= |\nabla^2 v(0)^2| 
\le |\nabla^2 u(0)^2|.
\]
This shows that the optimal constant $\kappa(p,n,d)$ has to be at least as large as $\kappa(p,n,n)$. 
\end{proof}

\begin{lemma}
\label{lem:Kato-optimal}
The optimal Kato constant $\kappa(p,2,d)$ in \eqref{eq:df-optimal-Kato} equals $\min(2, 1+(p-1)^2)$ for any $p \ge 1$ and $d \in \N$. 
\end{lemma}

\begin{proof}
We follow the proof strategy of \Cref{lem:ingredients}\ref{it:mixedC-S-K}, suitably adapting some technical details. 

We may choose an orthonormal basis of $\R^2$ for which the partial derivatives $u_1,u_2$ are orthogonal, and then choose an orthonormal basis of $\R^d$ in which $u_1 = (\alpha_1,0,\ldots,0)\in\R^d$ and $u_2 = (0,\alpha_2,0,\ldots,0)\in\R^d$. Denoting the desired constant as $\kappa(p) := \min(2, 1+(p-1)^2)$, the Kato inequality $\kappa(p) |\nabla |\nabla u|(0)|^2 \le |\nabla^2 u(0)|^2$ decouples into two separate inequalities: 
\begin{equation}\label{eq:kato-splitintotwo}
\begin{split}
\kappa(p) (\alpha_1 u_{11}^1 + \alpha_2 u_{12}^2)^2 
&\le (u_{11}^1)^2 + 2 (u_{12}^2)^2 + (u_{22}^1)^2,\\
\kappa(p)(\alpha_2 u_{22}^2 + \alpha_1 u_{12}^1)^2 
&\le (u_{22}^2)^2 + 2 (u_{12}^1)^2 + (u_{11}^2)^2.
\end{split}
\end{equation}
The (non-negative) terms $\sum_{i,j=1}^2\sum_{k=3}^d (u_{ij}^k)^2$ of $|\nabla^2 u (0)|^2$ will not appear on the right-hand side of the inequality \eqref{eq:df-optimal-Kato}. For the sake of clarity, we shall focus on proving the first inequality of \eqref{eq:kato-splitintotwo}. 

Note that the first equation of the $p$-harmonic system
\[
0 = u_{11}^1+u_{22}^1 + (p-2)\alpha_1(\alpha_1 u_{11}^1 + \alpha_2 u_{12}^2) 
\]
can be used to express $u_{22}^1$ in terms of $u_{11}^1,u_{12}^2$. Introducing auxiliary variables $x := u_{11}^1$, $y = u_{12}^2$ and then translating to rotated variables $z = \alpha_1 x + \alpha_2 y$, $w = \alpha_2 x - \alpha_1 y$, we transform the desired inequality to the form 
\begin{align*}
\kappa(p) z^2 
& \le x^2 + 2y^2 + (x+(p-2)\alpha_1 z)^2 \\
& = 2z^2 + 2w^2 - x^2 + (x+(p-2)\alpha_1 z)^2 \\
& = (2+p(p-2)\alpha_1^2) z^2 + 2(p-2)\alpha_1\alpha_2 zw + 2w^2.
\end{align*}
Now we need to distinguish between the two cases: 

\textsc{Case 1: $p > 2$, $\kappa(p) = 2$.} Here, the inequality simplifies further to 
\[
0 \le p\alpha_1^2 z^2 + 2\alpha_1\alpha_2 zw + \frac{2}{p-2} w^2.
\]
As a quadratic form in $z$ and $w$, the right-hand side has the discriminant $4 \alpha_1^2 (\alpha_2^2-\frac{2p}{p-2})$, which is nonpositive for all $\alpha_1,\alpha_2 \in [0,1]$ and $p > 2$. Hence, the inequality always holds. 

\textsc{Case 2: $1 \le p \le 2$, $\kappa(p) = 1+(p-1)^2$.} The inequality takes the form 
\[
(1+(p-1)^2) z^2 \le (2+p(p-2)\alpha_1^2) z^2 + 2(p-2)\alpha_1\alpha_2 zw + 2w^2, 
\]
which is easiest to analyse as a quadratic inequality in $p$. Indeed, it can be rewritten as $(1-\alpha_1^2)z^2 \cdot p^2 + ap + b \le 0$ (for some coefficients $a,b$). Since the leading coefficient $(1-\alpha_1^2)z^2$ is nonnegative, by convexity it is enough to consider the edge cases: 
$p=1$ and $p=2$. If only one recalls the equality $\alpha_1^2+\alpha_2^2 = 1$, both cases are trivial to check (the first one reduces to $(\alpha_2 z - \alpha_1 w)^2 + (1+\alpha_2^2)w^2 \ge 0$, the second one to $2w^2 \ge 0$). This finishes the proof. 
\end{proof}

\begin{example}
\label{ex:Kato-example}
For $p = 1+\sqrt{2}$ we have 
\[
\kappa(p,3,2) \le \frac{25+2\sqrt{2}}{14} \approx 1.9877 < 2 = \kappa(p,3,1),
\]
as seen on the following explicit example: 
\[
\nabla u(0) = 
\begin{bmatrix}
(1,0) \\
(0,1) \\
(0,0)
\end{bmatrix}, \quad
\nabla^2 u(0) = 
\begin{bmatrix}
(\alpha,0) & (0,\beta) & (0,0) \\
(0,\beta) & (\gamma,0) & (0,0) \\
(0,0) & (0,0) & (\gamma,0)
\end{bmatrix}
\]
with $\alpha = -15+28 \cdot \sqrt{2}$, $\beta = 1+14 \cdot \sqrt{2}$, $\gamma = -17$. Indeed, one can check that the $p$-harmonic system reduces to $\tfrac{p}{2} \alpha + \tfrac{p-2}{2} \beta + 2 \gamma = 0$, while $|\nabla |\nabla u|(0)|^2 = \frac{(\alpha+\beta)^2}{2}$ and $|\nabla^2 u(0)|^2 = \alpha^2 + 2\beta^2 + 2\gamma^2$.
\end{example}

Let us end this section by explaining how can this example (and similar ones) be obtained. If one fixes a unit vector $v \in \R^n \otimes \R^d$, then the optimization problem in \eqref{eq:elementary-optimal} is simply a quadratic optimization problem in terms of $w$, which can be solved using standard methods (either symbolically or numerically). For $\nabla u(0)$ as in Example \ref{ex:Kato-example} and $v := \frac{\nabla u(0)}{|\nabla u(0)|}$, the optimal $w \in \textrm{Sym}_n(\R^d)$ is the one described above. 

However, even for this particular choice of $p = 1+\sqrt{2}$, the optimal constant $\kappa(p,3,2)$ is actually slightly smaller than $\frac{25+2\sqrt{2}}{14} \approx 1.9877448$. By searching randomly through the space of possible $v$ (and optimizing $w$ for each choice), we found that $\kappa(p,3,3) \approx \kappa(p,3,2) \approx 1.9876817$. Moreover, our simulations suggest that the gap phenomenon $\kappa(p,n,2) < \kappa(p,n,1)$ persists in higher dimensions $n$, but it is only significant for $p$ near the critical exponent $p(n) = 1+\sqrt{n-1}$. We call it \emph{critical} because $\kappa(p,n,1) = 1+\frac{(p-1)^2}{n-1}$ for $p \le p(n)$ and $\kappa(p,n,1) = 2$ for $p \ge p(n)$, and the Kato inequality problem is substantially different in these two regimes. 

\appendix
\section{Explicit constant estimates}\label{s:constants}
In this section we derive the precise constant estimates that we were using in Step 3 of the proof of \Cref{th:regularity-near-n-1}.

\subsection{Hardt and Lin extension theorem with explicit constant}
We obtain an estimate of the constant in \cite[Theorem 6.2]{HLp} (see also \cite[p.556]{HardtKinderlehrerLin1986} and \cite[Lemma A.1]{HardtKinderlehrerLinstable}). We mention here that the idea of using retractions with analytic estimates in the proof of \cite[Theorem 6.2]{HLp} is due to Federer and Fleming \cite[Section 5 and 6]{FedererFleming} and was used by them in the proof of the isoperimetric inequality. 
\begin{theorem}\label{th:HL-constantestimate}
 Let $1<p< 4$ and $v\in W^{1,p}(B^3,\R^{4})$ be such that $|v(x)| = 1$ for a.e. $x\in \partial B^3$. Then there exists a map $u\in W^{1,p}(B^3,\S^3)$ such that 
 \[
  u\big\rvert_{\partial B^3} = v\big\rvert_{\partial B^3}
 \]
and 
\begin{equation}\label{eq:HL-constant}
 \int_{B^3} |\nabla u|^p \dif x \le C_{\HL}(p) \int_{B^3} |\nabla v|^p \dif x,
\end{equation}
where $C_{\HL}(p)=\frac{1}{(4-p) V_p(r)}$ and $V_p \coloneqq \int_0^1 t^3 (1+t)^{-p} \dif t$. In particular, for all $p\in[2,3]$ we have $C_{\HL}(p)\le C_{HL}(3) = \frac{8}{17-24\ln 2}$.
\end{theorem}

\begin{remark}
An even more explicit, but slightly worse, constant in \eqref{eq:HL-constant} is given by $\widetilde{C}_{\HL}(p) = \frac{2^{p+2}}{4-p}$, see \eqref{eq:worseconstant}.
\end{remark}

\begin{proof}[Proof of \Cref{th:HL-constantestimate}]
 Let $a$ be a point in the $4$-dimensional ball $B_{1}^{4}$. We consider the family of maps
 \begin{equation}\label{eq:definicjaua}
  \tilde u_a(x) = \frac{v(x) - a}{|v(x)-a|}, \quad x\in B^3.
 \end{equation}
For any $i\in\{1,2,3\}$ we compute 
\begin{equation}
\pl_{x_i} \tilde u_a 
= |v-a|^{-1} \pl_{x_i} v - |v-a|^{-3} \la v-a, \pl_{x_i} v \ra (v-a),
\end{equation}
which gives
\begin{equation}\label{eq:partialderivativessquared}
 \begin{split}
|\pl_{x_i} \tilde u_a|^2 &= |v-a|^{-2} |\pl_{x_i} v|^2 + |v-a|^{-6} |\la v-a, \pl_{x_i} v \ra|^2 |v-a|^2 - 2 |v-a|^{-4} |\la v-a, \pl_{x_i} v \ra|^2 \\
&= |v-a|^{-2} |\pl_{x_i} v|^2 - |v-a|^{-4} |\la v-a, \pl_{x_i} v \ra|^2 \\
&\le |v-a|^{-2} |\pl_{x_i} v|^2.
 \end{split}
\end{equation}
Summing up \eqref{eq:partialderivativessquared} over $i\in \{1,2,3\}$ we obtain $|\nabla \tilde u_a|^2 \le |v-a|^{-2} |\nabla v|^2$ and thus
\begin{equation}\label{eq:nablauatop}
 |\nabla \tilde u_a(x)|^p
 \le \frac{|\nabla v(x)|^p}{|v(x)-a|^p}.
\end{equation}
Now we note that for any $z\in \R^{4}$, 
\begin{equation}\label{eq:anestimate}
 \int_{B_1^{4}}\frac{1}{|z-a|^p} \dif a 
 \le \int_{B_1^{4}}\frac{1}{|a|^p} \dif a 
 =  \frac{\abs{\S^3}}{4-p}.
\end{equation}
To see why the inequality above holds, one needs to write the two integrals in terms of measures of corresponding superlevel sets. One then observes a pointwise inequality: for each $t > 0$ the superlevel set $\{ x \in B_1 : |z-a|^{-p} > t \}$ has smaller measure than $\{ x \in B_1 : |a|^{-p} > t \}$. 

Combining \eqref{eq:nablauatop} with \eqref{eq:anestimate} we obtain
\begin{equation}\label{eq:integraloversomething}
 \int_{B_1^{4}} \int_{B^3} |\nabla \tilde u_a|^p \dif x \dif a 
 \le \int_{B^3} \int_{B_1^{4}} \frac{|\nabla v(x)|^p}{|v(x)-a|^p} \dif a \dif x 
 \le \frac{\abs{\S^3}}{4-p} \int_{B^3} |\nabla v|^p \dif x.
\end{equation}
Since $|B^4_1| = \frac 14 \abs{\S^3}$, we can choose a point $a_0\in B_{1}^{4}$ for which 
\begin{equation}\label{eq:choiceofa0}
 \int_{B^3} |\nabla \tilde u_{a_0}|^p \dif x 
 \le \frac{\abs{\S^3}}{4-p} |B^4_1|^{-1} \int_{B^3} |\nabla v|^p \dif x 
 = \frac{4}{4-p} \int_{B^3} |\nabla v|^p \dif x.
\end{equation}
When restricted to $\partial B^3$, the map $\tilde u_{a_0}$ constructed above is $ P_{a_0} \circ v$, where 
\[
  P_{a_0}\colon \S^3\to \S^{3}, \qquad
  P_{a_0}(z)\coloneqq \frac{z - a_0}{|z -a_0|}. 
\]
This differs from $v|_{\pl B^3}$. To correct it, we note that $P_{a_0}\colon \S^3 \to \S^3$ is a smooth diffeomorphism (as long as $a_0$ lies in the interior of $B^4$) and so we can set 
\[
 \tilde u \coloneqq P_{a_0}^{-1} \circ \tilde u_{a_0}.
\]
Now $\nabla \tilde u = \nabla P_{a_0}^{-1} \cdot \nabla \tilde u_a$, so 
\begin{equation}\label{eq:estimatewithQa0}
|\nabla \tilde u| \le Q_{a_0} |\nabla \tilde u_{a_0}|, 
\end{equation}
where the constant $Q_{a_0}$ can be characterized as 
\[
Q_{a_0} \coloneqq \sup_{\substack{b \in \S^3, \\ v \in T_b \S^3}} \frac{|\nabla P_{a_0}^{-1}(b) \cdot v|}{|v|}
= \sup_{\substack{q \in \S^3, \\ w \in T_q \S^3}} \frac{|w|}{|\nabla P_{a_0}(q) \cdot w|} 
= \left( \inf_{\substack{q \in \S^3, \\ w \in T_q \S^3}} \frac{|\nabla P_{a_0}(q) \cdot w|}{|w|} \right)^{-1}.
\]
By a straightforward calculation we have 
\[
 |\nabla P_{a_0}(q) \cdot w|^2 = |q-{a_0}|^{-2} \left( |w|^2 - \left( \frac{q-{a_0}}{|q-{a_0}|} \cdot w \right)^2 \right),
\]
which can be interpreted as 
\[
\frac{|\nabla P_{a_0}(q) \cdot w|}{|w|}
= \frac{\cos(\alpha)}{|q-{a_0}|},
\]
with $\alpha$ denoting the angle between the vector $w$ and the subspace $(q-{a_0})^\perp$. For $w \in T_q \S^3$, the maximal such angle is the angle $\angle {a_0}q0$ in \Cref{picture}.

\begin{figure}
\begin{center}
\begin{tikzpicture}
\def\picR{2}
\def\nodeR{1.5pt}
\def\nodeA{-160}
\path 
	coordinate (o) at (0,0)
	coordinate (a) at (0,0.6*\picR)
	coordinate (q) at (\nodeA:\picR);
\draw (0,0) circle (\picR);
\begin{scope}
	\draw[clip] (o)--(q)--(a);
	\draw (q) circle (0.5*\picR);
\end{scope}
\node[above right = 0.2*\picR and 0.4*\picR of q] {$\alpha$};
\foreach \p in {o,a,q}
	\fill (\p) circle (\nodeR);
\node[right=\nodeR of o] {$0$};
\node[right=\nodeR of a] {$a_0$};
\node[below left=\nodeR of q] {$q$};
\draw[->] (q) -- ++({\nodeA-90}:1.3) node[left] {\tiny{$w\in T_q \S^3$}};
\end{tikzpicture}
\end{center}
\caption{Computation of $Q_{a_0}$}
\label{picture}
\end{figure}
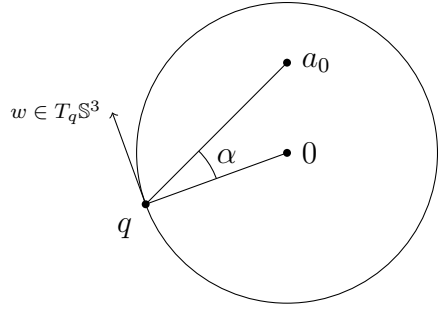
By the law of cosines, we have $\cos \alpha = \frac{|q-{a_0}|^2+1-|{a_0}|^2}{2|q-{a_0}|}$ and hence for fixed $q \in \S^3$, the minimal value of $\frac{\abs{P_{a_0}(q) \cdot w}}{|w|}$ is $\frac{|q-{a_0}|^2+1-|{a_0}|^2}{2|q-{a_0}|^2}$. Note that this is a decreasing function of $|q-{a_0}|$, so its minimum over $q \in \S^3$ is achieved when $q = -\frac{{a_0}}{|{a_0}|}$ and $|q-{a_0}|=1+|{a_0}|$. Thus, 
\begin{equation}\label{eq:Qa0estimate}
Q_{a_0} = \left( \frac{(1+|{a_0}|)^2+1-|{a_0}|^2}{2(1+|{a_0}|)^2} \right)^{-1} = 1+|{a_0}|.
\end{equation}
Since $a_0\in B_1^4$, we have $1+|{a_0}|\le 2$. Finally, from \eqref{eq:estimatewithQa0}, \eqref{eq:choiceofa0}, and \eqref{eq:Qa0estimate} we have
\begin{equation}\label{eq:worseconstant}
 \int_{B^3} |\nabla \tilde u|^p \dif x 
 \le \frac{2^{p+2}}{4-p} \int_{B^3} |\nabla v|^p \dif x
\end{equation}
and thus, we obtain inequality \eqref{eq:HL-constant} with a constant
\begin{equation}\label{eq:worseconstantdef}
\widetilde C_{\HL}(p) = \frac{2^{p+2}}{4-p}, 
\quad \text{in particular, } \quad 
\widetilde C_{\HL}(2) = 8, \quad \widetilde C_{\HL}(3) = 32.
\end{equation}

However, the constant $\widetilde C_{\HL}(p)$ may be still improved with only a little bit more care. Let us denote 
 \[
  u_a(x) = P_{a}^{-1} \left( \frac{v(x) - a}{|v(x)-a|} \right), \quad x\in B^3.
 \]
This map already agrees with $v$ on $\S^3$. According to previous calculations, i.e., \eqref{eq:nablauatop}, \eqref{eq:estimatewithQa0}, and \eqref{eq:Qa0estimate} we have the pointwise bound 
\[
|\nabla u_a(x)| 
\le (1+|a|)  \frac{|\nabla v(x)|}{|v(x)-a|}.
\]
Introducing the weight $\omega(a) \coloneqq (1+|a|)^{-p}$ to eliminate the power of $1+|a|$ appearing in the calculations, we then have (arguing as in \eqref{eq:integraloversomething})
\begin{equation}
 \int_{B_1^{4}} \int_{B^3} |\nabla u_a|^p \dif x\, \omega(a) \dif a 
 \le \int_{B^3} \int_{B_1^{4}} \frac{|\nabla v(x)|^p}{|v(x)-a|^p} \dif a \dif x 
 \le \frac{|\S^3|}{4-p} \int_{B^3} |\nabla v|^p \dif x.
\end{equation}
Denoting $V_p \coloneqq \int_0^1 t^3 (1+t)^{-p} \dif t$, we can now choose a point $a_0 \in B_{1}^{4}$ for which 
\begin{equation}
 \int_{B^3} |\nabla u_{a_0}|^p \dif x 
 \le C_{\HL}(p) \int_{B^3} |\nabla v|^p \dif x, 
\end{equation}
where $C_{\HL}(p)=\frac{1}{(4-p) V_p}$, which is clearly increasing in $p$. For $p=2,3$ the explicit formulas for $V_2$ and $V_3$ can be found. We obtain
\begin{equation}
C_{\HL}(2) = \frac{1}{6\ln 2 - 4} \approx 6.29,
\qquad
C_{\HL}(3) = \frac{8}{17-24\ln 2} \approx 21.95.
\end{equation}
\end{proof}

\begin{remark}
 Restricting $a$ to some smaller ball $B_r^4$ instead of $B_1^4$ in \eqref{eq:definicjaua} does not lead to an improved constant.
\end{remark}

\subsection{Poincar\'{e} inequality on a sphere with explicit constant.}
From \cite{Valtorta12} we have the following. 
\begin{theorem}[\cite{Valtorta12}]\label{th:ValtortaPoincare}
 If $f\in W^{1,p}(\S^2)$ satisfies $\int_{\S^2}|f|^{p-2}f=0$, then
 \begin{equation}\label{eq:poincare-scalar}
  \int_{\S^2}|f|^p \le C_{P:sc}(p) \int_{\S^2} |\nabla f|^p,
 \end{equation}
 where $C_{P:sc}(p)=\frac{p^p \sin^p(\frac{\pi}{p})}{2^p(p-1)}$.
\end{theorem}

\begin{corollary}\label{cor:vectorialpoincare}
 Let $u\in W^{1,p}(\S^2,\S^3)$ for $p\in[2,3]$, then there is $\bar u \in \R^4$ such that
 \begin{equation}\label{eq:poincare-vectorial}
  \int_{\S^2} |u - \bar u|^p \le C_{P}(p) \int_{\S^2} |\nabla u|^p ,
 \end{equation}
 where $C_P(p) = \frac{p^p \sin^p(\frac{\pi}{p})}{4(p-1)}$. In particular $\sup_{p\in[2,3]} C_P(p) = C_{P}(3) =\frac{3^{9/2}}{2^6}$.
\end{corollary}

\begin{proof}
Since $\S^3\subset\R^4$ we write $u=(u^1,u^2,u^3,u^4)$. It is not difficult to see that for each $i \in \{1,2,3,4\}$ there is $\bar u^i \in \R$ which minimizes $\int_{\S^2} |u^i - \bar u^i|^p$, and for such $\bar u^i$ we have $\int_{\S^2} |u^i - \bar u^i|^{p-2}(u^i - \bar u^i)=0$.

We have by \eqref{eq:poincare-scalar} on each component
\[
 \begin{split}
  \int_{\S^2} |u - \bar u|^p &= \int_{\S^2} \brac{\sum_{i=1}^4(u^i - \bar u^i)^2}^{\frac p2} \le 4^{\frac p2 -1} \sum_{i=1}^4 \int_{\S^2} |u^i - \bar u^i|^p\\
  &\le 2^{p -2}C_{P:sc}(p)\sum_{i=1}^4 \int_{\S^2} |\nabla u^i|^p \le 2^{p-2}C_{P:sc}(p)\int_{\S^2}|\nabla u|^p,
 \end{split}
\]
where in the last inequality we used the elementary inequality $\|x\|_{\ell^p}\le \|x\|_{\ell^2}$ for $p\ge 2$.
\end{proof}

\subsection{Trace inequality with explicit constant}
We will establish a basic trace inequality 
\[
\forall \, u \colon \S^2 \to \R^4 \quad \exists \, v \colon \B^3 \to \R^4 \qquad 
v |_{\S^2} = u, \quad
\int_{\B^3} |\nabla v|^p \le C_T \int_{\S^2} |\nabla u|^p.
\]
with the constant $C_T$ in some explicit form.

\begin{theorem}\label{th:trace}
Assume that $p\in[2,3]$ and $u\in W^{1,p}(\S^2,\S^3)$. Then there exists a map $v\in W^{1,2}(B^3,\R^4)$ such that $v\big\rvert_{\partial B^3}= u$ in the trace sense and
\begin{equation}\label{eq:traceinequality}
\int_{B^3}|\nabla v|^p\le C_T(p) \int_{\S^2} |\nabla u|^p,
\end{equation}
where $C_T(p) = 2^{(p-2)/2} \frac{1+\alpha^p C_{P}(p)}{3+p(\alpha-1)}$, $\alpha = \frac{2^{5/3}}{3^{3/2}}$, and $C_{P}(p)$ is the constant from the Poincar\'{e} inequality on the sphere $\S^2$ ---  \Cref{cor:vectorialpoincare}. In particular, for all $p\in [2,3]$ we have $C_T(p)\le C_T(3) = \frac{3^{3/2}}{2^{13/6}}$.
\end{theorem}

\begin{proof}
For $\alpha > 0$ to be chosen later, let us define 
 \[
 v(x) \coloneqq |x|^\alpha u(x/|x|) + (1-|x|^\alpha) \ov{u}
 \quad \text{for all } x\in B^3,
 \]
where $\ov{u}\in\R^4$ is the constant required in the Poincar\'{e} inequality, see \Cref{cor:vectorialpoincare}. The map $v$ interpolates between $u(x)$ and $\ov{u}$ on each ray $\{ tx : t \in [0,1] \}$. We compute the energy of $v$:
\begin{align*}
|\nabla v(x)|^2 
& = |x|^{2(\alpha-1)} \left( \alpha^2 |u(x/|x|)-\ov{u}|^2 + |\nabla u(x/|x|)|^2 \right),
\end{align*}
hence, 
\begin{equation}\label{eq:tracewithalpha}
\begin{split}
\int_{\B^3} |\nabla v(x)|^p \dx 
& \le 2^{(p-2)/2} \int_{\B^3} |x|^{p(\alpha-1)} \brac{\alpha^p |u(x/|x|)-\ov{u}|^p + |\nabla u(x/|x|)|^p} \dx \\
& = 2^{(p-2)/2} \int_0^1 r^{2+p(\alpha-1)} \dd r \int_{\S^2} \brac{\alpha^p |u(\xi)-\ov{u}|^p + |\nabla u(\xi)|^p} \dif \mathcal H^2(\xi) \\
& \le 2^{(p-2)/2} \frac{1+\alpha^p C_{P}(p)}{3+p(\alpha-1)} \int_{\S^2} |\nabla u(\xi)|^p \dif \mathcal H^2(\xi). 
 \end{split}
\end{equation}
Let us denote $C_T(p,\alpha)\coloneqq 2^{(p-2)/2} \frac{1+\alpha^p C_{P}(p)}{3+p(\alpha-1)}$ for $p\in[2,3]$ and $\alpha\in(0,1]$. It can be checked that $C_T(p,\alpha)$ is an increasing function of $p$ and that $\inf_{\alpha\in(0,1]} C_T(3,\alpha)$ is achieved for $\alpha = \frac{2^{5/3}}{3^{3/2}}$. Hence, we conclude with this choice of $\alpha$ in \eqref{eq:tracewithalpha}.
\end{proof}

\end{document}